\documentclass[11pt]{amsart}
\usepackage{amscd}
\usepackage{amsfonts}
\usepackage{amsmath}
\usepackage{amssymb}
\usepackage{amsthm}
\usepackage{latexsym}
\usepackage{eucal}
\usepackage{url}
\usepackage{mathrsfs}
\usepackage{newcent}
\usepackage[all]{xy}
\usepackage{color}

\usepackage{pdfsync}

\let\ms\mathscr
\let\mf\mathfrak
\let\mc\mathcal
\let\ol\overline
\let\wh\widehat

\newcommand{\field}[1]{\mathbf #1}
\newcommand{\widebar}[1]{\overline{#1}}
\newcommand{\cris}{\mathrm{cris}}
\newcommand{\simto}{\stackrel{\sim}{\to}}
\newcommand{\shom}{\ms H\!om}
\newcommand{\m}{\boldsymbol{\mu}}

\newcommand{\spec}{\operatorname{Spec}}
\newcommand{\spf}{\operatorname{Spf}}
\newcommand{\et}{\operatorname{\acute{e}t}}
\newcommand{\Het}{\operatorname{H}_{\et}}
\newcommand{\Hcris}{\operatorname{H}_{\cris}}
\renewcommand{\H}{\operatorname{H}}
\newcommand{\Gal}{\operatorname{Gal}}

\newcommand{\inj}{\hookrightarrow}
\newcommand{\id}{\operatorname{id}}

\newcommand{\tw}{\text{\rm tw}}
\newcommand{\rshom}{\mathbf{R}\shom}

\newcommand{\Hdr}{\operatorname{H}_{\rm dR}}

\newcommand{\R}{\field R}
\newcommand{\C}{\field C}

\newcommand{\Z}{\field Z}
\newcommand{\Q}{\field Q}

\renewcommand{\L}{\field L}

\newcommand{\G}{\field G} 

\DeclareMathOperator{\Pic}{Pic}
\DeclareMathOperator{\pr}{pr}

\DeclareMathOperator{\tr}{tr}

\DeclareMathOperator*{\tensor}{\otimes}
\DeclareMathOperator*{\ltensor}{\stackrel{\field L}{\otimes}}
\DeclareMathOperator{\rk}{\operatorname{rk}}
\DeclareMathOperator{\Hom}{\operatorname{Hom}}
\DeclareMathOperator{\Ext}{\operatorname{Ext}}
\DeclareMathOperator{\Br}{\operatorname{Br}}
\DeclareMathOperator{\NS}{NS}

\DeclareMathOperator{\Chow}{CH}
\DeclareMathOperator{\D}{D}
\DeclareMathOperator{\disc}{disc}
\DeclareMathOperator{\Sh}{\mathbf{Sh}}
\DeclareMathOperator{\mSh}{Sh}
\DeclareMathOperator{\chern}{ch}
\DeclareMathOperator{\Todd}{Td}
\DeclareMathOperator{\Aut}{Aut}

\DeclareMathOperator{\Def}{Def}

\theoremstyle{plain}
\newtheorem{lem}{Lemma}[subsection]
\newtheorem{thm}[lem]{Theorem}

\newtheorem*{mainthm}{Main Theorem}
\newtheorem{prop}[lem]{Proposition}

\newtheorem{cor}[lem]{Corollary}
\newtheorem*{corollary}{Corollary}

\theoremstyle{definition}
\newtheorem{defn}[lem]{Definition}

\theoremstyle{remark}
\newtheorem{remark}[lem]{Remark}

\newtheorem{notation}[lem]{Notation}

\title[Endlichkeit der K3-Fl\"achen und die Tate Vermutung]{Finiteness of K3 surfaces and the Tate conjecture}
\author{Max Lieblich}
\email{lieblich@math.washington.edu}
\author{Davesh Maulik}
\email{dmaulik@math.columbia.edu}
\author{Andrew Snowden}
\email{asnowden@math.mit.edu}

\date{May 21, 2013}

\begin{document}

\begin{abstract}
Given a finite field $k$ of characteristic $p \geq 5$, we show that
the Tate conjecture holds for K3 surfaces over $\ol{k}$ if and only if
 there are only finitely many K3 surfaces defined over each finite extension of $k$.
\end{abstract}

\maketitle
\tableofcontents

\section{Introduction}
\label{S:intro}

Given a class of algebraic varieties, it is reasonable to ask if there are only finitely many members defined over a
given finite field.  While this is clearly the case when the appropriate moduli functor is bounded, matters are often
not so simple.  For example, consider the case of abelian varieties of a given dimension $g$.  There
is no single moduli space parameterizing them; rather, for each integer $d \ge 1$ there is a moduli space
parameterizing abelian varieties of dimension $g$ with a polarization of degree $d$.
It is nevertheless possible to show (see \cite[Theorem 4.1]{zarhin}, \cite[Corollary 13.13]{milne}) that there are only finitely many
abelian varieties over a given finite field, up to isomorphism.
Another natural class of varieties where this difficulty arises is the case of K3 surfaces.
As with abelian varieties, there is not a single moduli space but rather a moduli space for each even integer
$d \geq 2$, parametrizing K3 surfaces with a polarization of degree $d$.

In this paper, we consider the finiteness question for K3 surfaces over finite fields.  Given a K3 surface $X$ defined over a finite field $k$ of characteristic $p$, the Tate conjecture predicts that the natural map
$$\Pic(X) \otimes \Q_{\ell} \rightarrow \Het^2(X_{\ol{k}}, \Q_{\ell}(1))^{\Gal(\ol{k}/k)}$$
is surjective for $\ell \ne p$.
It admits many alternate formulations; for example, it is equivalent to the statement that the
Brauer group of $X$ is finite.  We say that \emph{$X/k$ satisfies the Tate conjecture over some extension $k'/k$
(resp.\ $\ol{k}$)\/} if the Tate conjecture holds for the base change $X_{k'}$ (resp.\ for all base changes $X_{k'}$
with $k'/k$ finite).

Our main result is that this conjecture is essentially equivalent to the finiteness of the set of K3 surfaces over $k$.
Precisely:

\begin{mainthm}
Let $k$ be a finite field of characteristic $p\geq 5$.
\begin{enumerate}
\item
There are only finitely many isomorphism classes of K3 surfaces over $k$ which satisfy the Tate conjecture
over $\ol{k}$.
\item If there are only finitely many isomorphism classes of K3 surfaces over the quadratic extension $k'$ of $k$ then
every K3 surface over $k$ satisfies the Tate conjecture over $k'$.
\end{enumerate}
In particular, if $p \ge 5$, the Tate conjecture holds for all K3 surfaces over $\ol{k}$ if and only if there are only finitely many
K3 surfaces defined over each finite extension of $k$.
\end{mainthm}

As the Tate conjecture is known for K3 surfaces of finite height in characteristic at least 5 \cite{NO}, we obtain
the following unconditional corollary:

\begin{corollary}
If $p \ge 5$ then there are only finitely many isomorphism classes of K3 surfaces of finite height defined over $k$.
\end{corollary}

Our argument proceeds as follows.
To obtain finiteness from Tate, it suffices to prove the existence of low-degree polarizations on K3 surfaces over $k$.  In order to do this, we use the Tate conjecture in both $\ell$-adic and crystalline cohomology to control the possibilities of the N\'eron-Severi lattice.  For the other direction, we use the finiteness statement and the existence of infinitely many Brauer classes to create a K3 surface with infinitely many twisted Fourier-Mukai partners.  Since this cannot happen in characteristic zero, we obtain a contradiction by proving a lifting result.
This argument does not rely on \cite{NO} (as it did in an earlier version of this paper).


\removelastskip\vskip.5\baselineskip\par\noindent{\bf Notation.}
Throughout, $k$ denotes a finite field of characteristic $p$ and cardinality $q=p^f$.  We fix an algebraic closure
$\ol{k}$ of $k$.

\removelastskip\vskip.5\baselineskip\par\noindent{\bf Acknowledgments.}
We would like to thank Jean-Louis Colliot-Th\'el\`ene, Daniel Huybrechts, Abhinav Kumar, Keerthi Madapusi Pera, Matthias Sch\"utt, Damiano Testa, Yuri Zarhin, and the referees for many helpful comments and discussions.  M.L.\
is partially supported by NSF grant DMS-1021444, NSF CAREER grant DMS-1056129, and the Sloan Foundation.  D.M.\ is partially supported by a
Clay Research Fellowship.

\section{Tate implies finiteness}
\label{sec2}

\subsection{Discriminant bounds for the \'etale and crystalline lattices}

In this section, we produce bounds on the discriminants of certain lattices constructed from the \'etale and
crystalline cohomologies of K3 surfaces over $k$.  We begin by recalling some terminology.
Let $A$ be a principal ideal domain.  By a \emph{lattice} over $A$, we mean a finite free $A$ module $M$ together with
a symmetric $A$-linear form $(,):M \otimes_A M \to A$.  We say that $M$ is \emph{non-degenerate} (resp.\
\emph{unimodular}) if the map $M \to \Hom_A(M, A)$ provided by the pairing is injective (resp.\ bijective).
The \emph{discriminant} of a lattice $M$, denote $\disc(M)$, is the determinant of the matrix $(e_i, e_j)$, where
$\{e_i\}$ is a basis for $M$ as an $A$-module; it is a well-defined element of $A/(A^{\times})^2$.  The lattice $M$ is
non-degenerate (resp.\ unimodular) if and only if its discriminant is non-zero (resp.\ a unit).  Note that the valuation
of $\disc(M)$ at a maximal ideal of $A$ is well-defined.

We will need a simple lemma concerning discriminants:

\begin{lem}
\label{disc-lem}
Let $A$ be a discrete valuation ring with uniformizer $t$.  Let $M$ be a lattice over $A$ and let $M' \subset M$ be an
$A$-submodule such that $M/M'$ has length $r$ as an $A$-module.  Regard $M'$ as a lattice by restricting the form from
$M$.  Then $\disc(M')=t^{2r} \disc(M)$ up to units of $A$.
\end{lem}

\begin{proof}
Let $e_1, \ldots, e_n$ be a basis for $M$ and let $f_1, \ldots, f_n$ be a basis for $M'$.  Let $B$ be the matrix
$(e_i, e_j)$ and let $B'$ be the matrix $(f_i, f_j)$.  Thus $\disc(M)=\det{B}$ and $\disc(M')=\det{B'}$.  Let
$C \in M_n(A)$ be the change of basis matrix, so that $f_i=C e_i$.  Then $\det(C)=t^r$ up to units of $A$.
As $B'=C^tBC$, the result follows.
\end{proof}

The following general result on discriminant bounds will be used several times in what follows.

\begin{prop}
\label{discrim}
Fix a positive integer $r$ and a non-negative even integer $w$.  There exist constants $C$ and $C'$, depending only on $r$, $w$ and $q$, with the following property.

Let $E$ be a finite unramified extension of $\Q_{\ell}$ with ring of integers $\mc{O}$.  Let $M$ be a lattice over $\mc{O}$ of rank $r$ equipped with an endomorphism $\phi$.  Let $v_0$ be the $\ell$-adic valuation of $\disc(M)$.  Assume that the characteristic polynomial of $\phi$ belongs to $\Z[T]$, that all eigenvalues of $\phi$ on $M[1/\ell]$ are Weil $q$-integers of weight $w$ and that $q^{w/2}$ is a semi-simple eigenvalue of $\phi$ on $M[1/\ell]$.
\begin{enumerate}
\item If $\ell>C$ then the discriminant $M^{\phi=q^{w/2}}$ has $\ell$-adic valuation at most $v_0$.
\item The discriminant of $M^{\phi=q^{w/2}}$ has $\ell$-adic valuation at most $C'+v_0$.
\end{enumerate}
\end{prop}

\begin{proof}
We first define the constants $C$ and $C'$.  Let $\mc{W}$ be the set of all Weil $q$-integers of weight $w$ and degree at most $r$.  It is easy to bound the coefficients
of the minimal polynomial of an element of $\mc{W}$, and so one sees that $\mc{W}$ is a finite set.  Let $S$ denote
the set of elements of $\Z[T]$ which are monic of degree $r$ and whose roots belong to $\mc{W}$.  Clearly, $S$ is
a finite set; enumerate its elements as $f_1(T), \ldots, f_m(T)$.  We can factor each $f_i(T)$ as $g_i(T) h_i(T)$,
where $g_i(T)$ is a power of $T-q^{w/2}$ and $h_i(T)$ is an element of $\Z[T]$ which does not have $q^{w/2}$ as a root.  For each $i$, pick rational polynomials $a_i(T)$ and $b_i(T)$ such that
\begin{displaymath}
a_i(T) g_i(T)+b_i(T) h_i(T)=1.
\end{displaymath}
Let $Q$ be the least common multiple of the denominators of the coefficients of $a_i(T)$ and $b_i(T)$.  Let $s$
be the maximal integer such that $\ell^s$ divides $Q$, for some prime $\ell$, and let $\ell_0$ be the largest prime
dividing $Q$.  We claim that we can take $C=\ell_0$ and $C'=2rs$.

We now prove these claims.  Thus let $M$ and $\phi$ be given, and put $N=M^{\phi=q^{w/2}}$.  The characteristic polynomial of $\phi$ belongs to $S$, and is thus equal to $f_i(T)$ for some $i$.  Put
$M_1=h_i(\phi) M$ and $M_2=g_i(\phi) M$.  One easily sees that $M_1 \oplus M_2$ is a finite index
$\mc{O}$-submodule of $M$ and that $M_1$ and $M_2$ are orthogonal.  Furthermore, $M_1$ is contained in $N$, since
$q^{w/2}$ is a semi-simple eigenvalue of $\phi$, and has finite index.

Suppose that $\ell>C$.  Then $a_i(T)$ and $b_i(T)$ belong to $\mc{O}[T]$ and so $M=M_1 \oplus M_2$.  Thus
$\disc(M)=\disc(M_1) \disc(M_2)$.  It follows that $\disc(M_1)$ has $\ell$-adic valuation at most $v_0$. As $N$ and $M_1$ are saturated in $M$ and $M_1 \subset N$, we have $N=M_1$, and so (1) follows.

Now suppose that $\ell$ is arbitrary.  Then $\ell^s a_i(T)$ and $\ell^s b_i(T)$ belong to $\mc{O}[T]$.
It follows that $M_1 \oplus M_2$ contains $\ell^s M$, and so $M/(M_1 \oplus M_2)$ has length at most $rs$ as an $\mc{O}$-module.  Lemma~\ref{disc-lem} shows that $\disc(M_1) \disc(M_2)$ divides $\ell^{2rs} \disc(M)$, and thus has $\ell$-adic valuation as most $2rs+v_0=C'+v_0$.  The lemma also shows that $\disc(N)$ divides $\disc(M_1)$, which proves (2).
\end{proof}

Let $X$ be a K3 surface over $k$.  For a prime number $\ell \ne p$ put
\begin{displaymath}
M_{\ell}(X)=\Het^2(X_{\ol{k}}, \Z_{\ell}), \qquad N_{\ell}(X)=M_{\ell}(X)^{\phi=q}.
\end{displaymath}
Then $M_{\ell}(X)$ is a free $\Z_{\ell}$-module of rank 22, and the cup product gives it the structure of a unimodular
lattice.  The space $M_{\ell}(X)$ admits a natural $\Z_{\ell}$-linear automorphism $\phi$, the geometric Frobenius
element of $\Gal(\ol{k}/k)$.  The map $\phi$ does not quite preserve the form, but satisfies
$(\phi x, \phi y)=q^2 (x, y)$.  It is known \cite{D} that the action of $\phi$ on $M_{\ell}(X)$ is semi-simple.
We give $N_{\ell}(X)$ the structure of a lattice by restricting the form from $M_{\ell}(X)$.

\begin{prop}
\label{step1a}
There exist constants $C_1=C_1(k)$ and $C_2=C_2(k)$ with the following properties.  Let $X$ be a K3 surface
over $k$ and let $\ell \ne p$ be a prime number.  Then
\begin{enumerate}
\item For $\ell>C_1$, the discriminant of $N_{\ell}(X)$ has $\ell$-adic valuation zero.
\item The discriminant of $N_{\ell}(X)$ has $\ell$-adic valuation at most $C_2$.
\end{enumerate}
\end{prop}

\begin{proof}
This follows immediately from Proposition~\ref{discrim} with $r=22$ and $w=2$, applied to $M=M_{\ell}(X)$.  Note that $v_0=0$.
\end{proof}

We also need a version of the above result at $p$.  Let $W=W(k)$ be the Witt ring of $k$.  Put
\begin{displaymath}
M_p(X)=\Hcris^2(X/W), \qquad N_p(X)=M_p(X)^{\phi_0=p}.
\end{displaymath}
Then $M_p(X)$ is a free $W$-module of rank 22, and the cup product gives it the structure of a unimodular
lattice.  The lattice $M_p(X)$ admits a natural semilinear automorphism $\phi_0$, the crystalline Frobenius.
The map $\phi=\phi_0^f$ is $W$-linear (where $q=p^f$).  We have $(\phi_0 x, \phi_0 y)=p^2 \phi_0((x, y))$.
(Note:  the $\phi_0$ on the right is the Frobenius on $W$.)
Since $\phi_0$ is only semi-linear, $N_p(X)$ is not a $W$-module, but a $\Z_p$-module.  We give $N_p(X)$ the
structure of a lattice via the form on $M_p(X)$.

We say that an eigenvalue $\alpha$ of a linear map is semi-simple if the $\alpha$-eigenspace coincides with the
$\alpha$-generalized eigenspace.  We now come to the main result at $p$:

\begin{prop}
\label{step1b}
There exists a constant $C_3=C_3(k)$ with the following property.  Let $X$ be a K3 surface over $k$.  Assume that
$q$ is a semi-simple eigenvalue of $\phi$ on $M_p(X)[1/p]$.  Then the discriminant of $N_p(X)$ has $p$-adic
valuation at most $C_3$.
\end{prop}

\begin{proof}
Let $X$ be given, and put $N'=M_p(X)^{\phi=q}$, so that $N_p(X)=(N')^{\phi_0=p}$.  Proposition~\ref{discrim}, with $r=22$ and $w=2$, bounds the $p$-adic valuation of $\disc(N')$ (as a lattice over $W$) in terms of $k$; in fact, the produced bound is the number $C_2$ from the previous proposition.  The following lemma (which defines
a constant $C_4$) now shows that the $p$-adic valuation of  $\disc(N_p(X))$ is bounded by $C_3=C_2f+44 C_4f$.
\end{proof}

\begin{lem}
There exists a constant $C_4=C_4(k)$ with the following property.  Let $N'$ be a lattice over $W$ of rank $n$.
Let $\phi_0$ be a semi-linear endomorphism of $N'$ satisfying $\phi_0^f=q$ and $(\phi_0 x, \phi_0 y)
=p^2 \phi_0((x,y))$, and put $N=(N')^{\phi_0=p}$.  Then $v_p(\disc(N))$ is at most $f v_p(\disc(N'))+2C_4 nf$.
\end{lem}

\begin{proof}
Let $g(T)=T-p$ and let $h(T)=(T^f-q)/(T-p)$, two monic polynomials in $\Z_p[T]$.  Pick polynomials $a(T)$ and $b(T)$ in
$\Q_p[T]$ such that
\begin{displaymath}
a(T) g(T) + b(T) h(T) = 1.
\end{displaymath}
Let $r$ be such that $p^r a(T)$ and $p^r b(T)$ belong to $\Z_p[T]$.  We claim that we can take $C_4=r$.

Thus let $N'$ with $\phi_0$ be given.  Let $L'$ be $N'$, regarded as a $\Z_p$-lattice with pairing $\langle, \rangle$
given by $\langle x, y \rangle =\tr_{W/\Z_p} (x, y)$.  We have $\rk(L')=fn$ and $\disc(L')
=\mathbf{N}(\disc(N'))$, where $\mathbf{N}$ is the norm from $W$ to $\Z_p$.  We regard $\phi_0$ as a linear map
of $L'$.  As such, it is semi-simple (when $p$ is inverted) with minimal polynomial $T^f-q$.  The identity
$\langle \phi_0 x, \phi_0 y \rangle=p^2 \langle x, y \rangle$ holds.
Let $L$ be the $p$ eigenspace of $\phi_0$, regarded as a sublattice of $L'$.  Then
$L$ and $N$ are the same subset of $N'$, but the form on $L$ is that on $N$ scaled by $f$.  It follows that
$\disc(N)$ divides $\disc(L)$.

Put $L_1=h(\phi_0) L'$ and $L_2=g(\phi_0) L'$.  Then $L_1$ and $L_2$ are orthogonal under $\langle, \rangle$ and
$L_1 \oplus L_2$ has index
at most $p^{fnr}$ in $L'$.  It follows from Lemma~\ref{disc-lem} that $\disc(L_1)$ divides $p^{2fnr} \disc(L')$.  Since
$L_1$ is contained in $L$, we find that $\disc(L)$ divides $p^{2fnr} \disc(L')$ as well.  Finally, we see that
$\disc(N)$ divides $p^{2fnr} \mathbf{N}(\disc(N'))$, and so $v_p(\disc(N))$ is bounded by $2fnr+f v_p(\disc{N'})$.
\end{proof}

\begin{remark}
As far as we are aware, it is not known if the action of $\phi$ on $M_p(X)[1/p]$ is semi-simple.  See
Corollary 7.5 of \cite{Ogus2} for a partial result.  However, for K3 surfaces which satisfy the Tate conjecture, semi-simplicity
is known, and can be easily be deduced from the aforementioned result of Ogus.
\end{remark}

\subsection{Controlling the N\'eron--Severi lattice}
\label{s:neron}

Let $X$ be a K3 surface over $k$.  Write $\NS(X)$ for the N\'eron--Severi group of $X$, which is a lattice (over
$\Z$) under the intersection pairing.  The main result of this section is the following:

\begin{prop}
\label{step2}
There exists a finite set $\ms{L}=\ms{L}(k)$ of lattices over $\Z$ with the following property:  if $X$ is a
K3 surface over $k$ which satisfies the Tate conjecture (over $k$) then $\NS(X)$ is isomorphic to a member of $\ms{L}$.
\end{prop}

\begin{lem}
\label{tate-ell}
Let $X$ be a K3 surface over $k$ which satisfies the Tate conjecture and let $\ell \ne p$ be a prime.  Then the map
\begin{displaymath}
c_1:\NS(X) \otimes \Z_{\ell} \to N_{\ell}(X)(1)
\end{displaymath}
is an isomorphism.
\end{lem}

\begin{proof}
The Kummer sequence on $X_{\ol{k}}$ gives a short exact sequence
\begin{displaymath}
0 \to \NS(X_{\ol{k}}) \otimes \Z/\ell^n \Z \to \H^2_{\et}(X_{\ol{k}}, \Z/\ell^n \Z(1)) \to \Br(X_{\ol{k}})[\ell^n] \to 0.
\end{displaymath}
Taking the inverse limit over $n$, we obtain an exact sequence
\begin{displaymath}
0 \to \NS(X_{\ol{k}}) \otimes \Z_{\ell} \to \H^2_{\et}(X_{\ol{k}}, \Z_{\ell}(1)) \to T_{\ell}(\Br(X_{\ol{k}})) \to 0,
\end{displaymath}
where $T_{\ell}$ here denotes Tate module.  The first map above is $c_1$.  Since the $\ell$-adic Tate module of any abelian group has no $\ell$-torsion, we see that the image of $\NS(X_{\ol{k}}) \otimes \Z_{\ell}$ in $\H^2_{\et}(X_{\ol{k}}, \Z_{\ell}(1))$ is saturated.  As $\NS(X)$ is saturated in $\NS(X_{\ol{k}})$, it follows that the image of the map
\begin{displaymath}
c_1:\NS(X) \otimes \Z_{\ell} \to N_{\ell}(X)(1)
\end{displaymath}
is saturated.  Since $X$ satisfies the Tate conjecture, the above map is an isomorphism when $\ell$ is inverted.  From saturatedness, it is therefore an isomorphism without $\ell$ inverted.
\end{proof}

\begin{remark}
The Chern class map appearing in the statement of the lemma may look odd, as $\NS(X)$ carries no Galois action
but it appears as if $N_{\ell}(X)(1)$ does.  However, one should regard $N_{\ell}(X)$ as carrying an action by
the inverse of the cyclotomic character (as it is a $\phi=q$ eigenspace), and so the Tate twist cancels this
action.
\end{remark}

In the following lemma and proof, let $\{1\}$ denote the Tate twist in crystalline and de Rham cohomology.

\begin{lem}
\label{tate-p}
Let $X$ be a K3 surface over $k$ which satisfies the Tate conjecture.  Then the map
\begin{displaymath}
c_1:\NS(X) \otimes \Z_p \to N_p(X)\{1\}
\end{displaymath}
is an isomorphism.
\end{lem}


\begin{proof}
Since $X$ satisfies the Tate conjecture, $c_1$ is an isomorphism after inverting $p$. It is therefore enough to show that the image of $\NS(X) \otimes \Z_p$ in $\Hcris^2(X/W)\{1\}$ is saturated, or equivalently, that
\begin{displaymath}
c_1 \colon \NS(X) \otimes \Z/p\Z \to \Hcris^2(X/W)\{1\} \otimes \Z/p\Z = \Hdr^2(X/k)\{1\}
\end{displaymath}
is injective. This follows from \cite[Remark~3.5]{D2}.
\end{proof}

We now return to the proof of the proposition.

\begin{proof}[Proof of Proposition~\ref{step2}]
Let $X$ be a K3 surface over $k$ satisfying the Tate conjecture.  Let $\ell \ne p$ be a prime number.  By Lemma~\ref{tate-ell},
$\NS(X) \otimes \Z_{\ell}$ is isomorphic, as a lattice, to $N_{\ell}(X)(1)$.  Since $N_{\ell}(X)(1)$ is isomorphic,
as a lattice, to $N_{\ell}(X)$,
we find that $\disc(\NS(X))$ and $\disc(N_{\ell}(X))$ have the same $\ell$-adic valuations.  Similarly, appealing
to Lemma~\ref{tate-p}, we
find that $\disc(\NS(X))$ and $\disc(N_p(X))$ have the same $p$-adic valuations.  Applying
Propositions~\ref{step1a} and~\ref{step1b}, we find that
$\vert \disc(\NS(X)) \vert$ is at most $p^{C_3} \prod_{\ell \le C_1} \ell^{C_2}$.  As there are only finitely many
isomorphism classes of lattices of a given rank and discriminant \cite[Ch.~9, Thm.~1.1]{Cassels}, it follows that there
are only finitely many possibilities for $\NS(X)$ (up to isomorphism).
\end{proof}

\subsection{Constructing low-degree ample line bundles}

We assume for the rest of \S \ref{sec2} that $p \ge 5$.
Let $K$ be the extension of $k$ of degree
\begin{displaymath}
6983776800 = 2^5\cdot 3^3\cdot 5^2\cdot 7\cdot 11\cdot 13\cdot 17\cdot 19.
\end{displaymath}
The field $K$ is relevant for the following reason (explained below):  if $X$ is a K3 surface over $k$ then any line
bundle on $X_{\ol{k}}$ descends to one on $X_K$.  The purpose of this section is to prove the following
proposition:

\begin{prop}
\label{step3}
There exists a constant $C_5=C_5(k)$ with the following property:  if $X$ is a K3 surface over $k$ which satisfies
the Tate conjecture over $K$ then $X$ admits an ample line bundle of degree at most $C_5$ defined over $K$.
\end{prop}

We begin with a lemma.

\begin{lem}
\label{step3-1}
Let $X$ and $X'$ be K3 surfaces over algebraically closed fields such that $\NS(X)$ and $\NS(X')$ are isomorphic as
lattices.  Then the set of degrees of ample line bundles on $X$ and $X'$ coincide.
\end{lem}

\begin{proof}
Put $N=\NS(X)$ and $N_{\R}=N \otimes \R$.  Let $\Delta$ be the set of elements $\delta$ in $N$ such that $(\delta,
\delta)=-2$.  For $\delta \in \Delta$, let $r_{\delta}:N_{\R} \to N_{\R}$ be the reflection given by $r_{\delta}(x)
=x+(x,\delta) \delta$.  Let $\Gamma$ be the group generated by the $r_{\delta}$, with $\delta \in \Delta$.
Finally, let $V$ be the set of elements $x$ in $N_{\R}$ such that $(x,x)>0$ and $(x, \delta) \ne 0$ for all $\delta
\in \Delta$.  Then $V$ is an open subset of $N_{\R}$ and the group $\pm \Gamma$ acts transitively on the set of
connected components of $V$ \cite[Prop.~1.10]{Ogus}.  Furthermore, there exists a unique connected component
$V_0$ of $V$ such that an element of $N$ is ample if and only if it lies in $V_0$ \cite[p.~371]{Ogus}.

Let $N'$, etc., be as above but for $X'$.  Choose an isomorphism of lattices $i:N \to N'$.  Clearly, $i$ induces
a bijection $\Delta \to \Delta'$ and a homeomorphism $V \to V'$.  Thus $i(V_0)$ is some connected component of
$V'$.  We can find an element $\gamma$ of $\pm \Gamma'$ such that $\gamma(i(V_0))=V_0'$.  Thus, replacing $i$ by
$\gamma i$, we can assume that $i(V_0)=V_0'$.  It then follows that $i$ induces a bijection between the set of
ample elements in $N$ and the set of ample elements in $N'$.  Since $i$ preserves degree, this proves the lemma.
\end{proof}

\begin{cor}
\label{step3-2}
For every lattice $L$ there is an integer $d(L)$ with the following property:  if $X$ is a K3 surface over an
algebraically closed field such that $\NS(X)$ is isomorphic to $L$ then $X$ admits an ample line bundle of degree
$d(L)$.
\end{cor}

We now return to the proof of Proposition~\ref{step3}.

\begin{proof}[Proof of Proposition~\ref{step3}]
Let $X$ be a K3 surface over $k$.  The Frobenius element $\phi$ of the absolute Galois group of $k$ acts on
$\NS(X_{\ol{k}})$ as a finite order endomorphism.  It is therefore semi-simple.  Furthermore, since its characteristic
polynomial has degree at most $22$ and has at least one eigenvalue equal to $1$, its remaining eigenvalues which are roots of
unity of degree at most $21$.
At this point, we use the defining property of $K$.  By construction, any integer $m$ with $\varphi(m) \leq 21$ must
divide $N = \deg[K:k]$.  It follows that $\phi^N=1$ holds.  In particular, we see that $\NS(X_{\ol{k}})=\NS(X_K)$.

Let $\ms{L}=\ms{L}(K)$ be the set of lattices provided by Proposition~\ref{step2}.  Let $C_5$ be the maximum
value of $d(L)$ for $L \in \ms{L}$, where $d(L)$ is as defined in Corollary~\ref{step3-2}.
Let $X$ be a K3 surface over $k$ satisfying the Tate conjecture over $K$.  Then $\NS(X_K)$, and thus
$\NS(X_{\ol{k}})$, belongs to $\ms{L}$.  It follows that $X_{\ol{k}}$, and thus $X_K$, admits an ample line bundle of
degree at most $C_5$.  This proves the proposition.
\end{proof}

\subsection{Finiteness of twisted forms}

Recall that a \emph{twisted form} of a K3 surface $X/k$ is a K3 surface over $k$ which is isomorphic to $X$ over $\ol{k}$.  The purpose of this section is to establish the following result.

\begin{prop}
\label{fin-twist}
Let $X$ be a K3 surface over the finite field $k$.  Then $X$ has only finitely many twisted forms, up to
isomorphism.
\end{prop}

The set of isomorphism classes of twisted forms of $X$ is in bijection with the non-abelian cohomology set
$\H^1(\Gal(\ol{k}/k),\Aut_{\ol{k}}(X_{\ol{k}}))$, and so to prove the proposition it suffices to show finiteness of this set.
We begin with two lemmas.  In what follows, $\wh{\Z}$ denotes the profinite completion of $\Z$ and $\phi$ a topological generator.  Suppose $\wh{\Z}$ acts continuously on a discrete group $E$.  A $1$-cocycle for this action is given by an element $x\in E$ such that $x x^{\phi} \cdots x^{\phi^{n-1}}=1$ for all sufficiently divisible integers $n$.  Moreover, cocycles represented by $x$ and $y$ are cohomologous if there exists an element $h \in E$ such that $x=h^{-1}yh^{\phi}$

\begin{lem}
Let $G$ and $G'$ be discrete groups on which $\wh{\Z}$ acts continuously and let $f:G \to G'$ be a $\wh{\Z}$-equivariant homomorphism
whose kernel is finite and whose image has finite index.  If $\H^1(\wh{\Z}, G')$ is finite then
$\H^1(\wh{\Z}, G)$ is finite.
\end{lem}
\begin{proof}
Let $x_1,\ldots,x_n\in G'$ be cocycles representing the elements of $\H^1(\wh{\Z},G')$ and $y_1,\ldots,y_j\in G'$ representatives for the right cosets of $f(G)$ in $G'$.
Given a cocycle $x\in G$, there is $i$ and $h\in G'$ such that
$$h^{-1}f(x)h^\phi=x_i.$$
Choosing $j$ and $z\in G$ such that $h=f(z)y_j$ yields
$$y_j^{-1}f(z)^{-1}f(x)f(z)^\phi y_j^\phi=x_i,$$
whence
$$f(z^{-1}xz^\phi)=y_jx_i(y_j^\phi)^{-1}.$$
Since $f$ has finite kernel, each element $y_jx_i(y_j^\phi)^{-1}$ has finitely many preimages in $G$, and we see that the union of this finite set of finite sets contains cocycles representing all of $\H^1(\wh{\Z},G)$, as desired.
\end{proof}

\begin{lem}
Let $G$ be a discrete group and let $\wh{\Z} \to G$ be a continuous homomorphism; regard $\wh{\Z}$ as acting on $G$ by
inner automorphisms via the homomorphism.  Assume that $G$ has only finitely many conjugacy classes of finite order
elements.  Then $H^1(\wh{\Z}, G)$ is finite.
\end{lem}

\begin{proof}
Let $g$ be the image of $\phi$ under the map $\wh{\Z} \to G$.
Continuity forces $g$ to have finite order.  Because $\wh{\Z}$ acts by conjugation, the cocycle condition for an element $x\in G$ simply amounts to $xg$ having finite order, and $x$ and $y$ are cohomologous if $xg$ and $yg$ are conjugate.  We
thus find that multiplication by $g$ gives a bijection between $\H^1(\wh{\Z}, G)$ and the set of conjugacy classes
of finite order elements of $G$.  This completes the proof.
\end{proof}

We now prove the proposition.

\begin{proof}[Proof of Proposition~\ref{fin-twist}]
Let $N=\NS(X_{\ol{k}})$ and let $N' \subset N_{\R}$ be the nef cone.  Let $G'$ be the group of automorphisms of the
lattice $N$ which map $N'$ to itself, let $G=\Aut_{\ol{k}}(X_{\ol{k}})$, let $G^\circ=\Aut_k(X_{\ol{k}})$, and let $\Gamma=\Aut(\ol{k}/k)
\cong \wh{\Z}$.  Since the natural action of $G^\circ$ on $N$ preserves $N'$, there are homomorphisms
$$\Gamma\to G^\circ\to G'.$$
In addition, conjugation by the image of $\phi$ in $G^\circ$ preserves $G$ and gives the natural Frobenius action on the automorphism group.    Thus, the natural map
$$f:G\to G'$$
is $\wh{\Z}$-equivariant with respect to the natural action on $G$ and the conjugation action on $G'$.
The map $f$ has finite kernel and its image has finite index by \cite[Thm~6.1]{LieblichMaulik}.
Furthermore, $G'$ has finitely many conjugacy classes of finite order elements (see \cite[\S 6]{Totaro}).  The
above two lemmas thus imply that $\H^1(\Gamma, G)$ is finite, which completes the proof.
\end{proof}

\subsection{Proof of Main Theorem (1)}\label{S:main thm 1}

We now complete the proof of the first part of the main theorem.  Let $M_d$ be the stack over $k$ of pairs $(X, L)$
where $X$ is a K3 surface and $L$ is a polarization of degree $d$.  It follows from Artin's representability theorem
that $M_d$ is Deligne--Mumford and locally of finite type over $k$; since the third power of any polarization is
very ample \cite{saint-donat}, the stack is of finite type.  Let $C_5$ be the constant produced by
Proposition~\ref{step3}.  Consider the diagram
\begin{displaymath}
\xymatrix{
{\displaystyle \coprod_{d=1}^{C_5} M_d(K)} \ar[r]^-{\alpha} &
\{ \textrm{isomorphism classes of K3's over $K$} \} \\
& \{ \textrm{isomorphism classes of K3's over $k$ satisfying Tate over $K$} \} \ar[u]_-{\beta} }
\end{displaymath}
By the definition of $C_5$, any element in the image of $\beta$ is also in the image of $\alpha$.  Since the
domain of $\alpha$ is finite, it follows that the image of $\beta$ is finite.  Any two elements of a fiber of
$\beta$ are twisted forms of each other, and so the fibers of $\beta$ are finite by Proposition~\ref{fin-twist}.  We
thus find that the domain of $\beta$ is finite, which completes the proof.

\section{Finiteness implies Tate}

\subsection{Twisted sheaves}

We use the notions and notation from \cite{L1}, \cite{L2}, and the references therein.  Recall the basic definition.  Fix a $\m_r$-gerbe over an algebraic space $\ms Z\to Z$.

\begin{defn}
 A sheaf $\ms F$ of (left) $\ms O_{\ms Z}$-modules is \emph{$\lambda$-fold $\ms Z$-twisted\/} if the natural left inertial $\m_r$-action $\m_r\times\ms F\to\ms F$ of a section $\rho$ of $\m_r$ is scalar multiplication by the $\rho^\lambda$.
\end{defn}

\begin{notation}
 Given a $\m_n$-gerbe $\ms Z\to Z$, write $\D^{\tw}(\ms Z)$ for the derived category of perfect complexes of $\ms Z$-twisted sheaves and $\D^{-\tw}(\ms Z)$ for the derived category of perfect complexes of $(-1)$-fold $\ms Z$-twisted sheaves.
\end{notation}

\subsection{$\ell$-adic $B$-fields}

Let $Z$ be a separated scheme of finite type over the field $k$.  The following is an ``$\ell$-adification'' of a notion familiar from mathematical physics.  For the most part, this rephrases well-known results in a form that aligns them with the literature on twisted Mukai lattices, to be developed $\ell$-adically in the next section.

\begin{defn}\label{D:B-field}
 An \emph{$\ell$-adic $B$-field on $Z$\/} is an element $$B\in\Het^2(Z,\Q_\ell(1)).$$
\end{defn}
We can write any $B$-field as $\alpha/\ell^n$ with $$\alpha\in\Het^2(Z,\Z_\ell(1))$$
a primitive element.  \emph{When we write $B$ in this form we will always assume (unless noted otherwise) that $\alpha$ is primitive\/}.

\begin{defn}
 Given a $B$-field $\alpha/\ell^n$ on $Z$, the \emph{Brauer class associated to $B$\/} is the image of $\alpha$ under the map
 $$\Het^2(Z,\Z_\ell(1))\to\Het^2(Z,\m_{\ell^n})\to\Br(Z)[\ell^n].$$
\end{defn}

\begin{notation}\label{N:n}
 Given $\alpha\in\Het^2(Z,\Z_\ell(1))$ we will write $\alpha_n$ for the image in $\Het^2(Z,\m_{\ell^n})$.  We will use brackets to indicate the map $\Het^2(Z,\m_{\ell^n})\to\Br(X)$.
\end{notation}
Thus, the Brauer class associated to the $B$-field $\alpha/\ell^n$ is written $[\alpha_n]$.

\begin{lem}\label{L:reduction multiplication}
 Given $\alpha\in\Het^2(Z,\Z_\ell(1))$ and positive integers $n, n'$, we have that
 $$\ell^n[\alpha_{n+n'}]=[\alpha_{n'}]\in\Br(Z).$$
\end{lem}
\begin{proof}
 Consider the commutative diagram of Kummer sequences
 $$\xymatrix{1\ar[r] & \m_{\ell^{n+n'}}\ar[r]\ar[d]_{\ell^n} & \G_m\ar[r]^{\ell^{n+n'}}\ar[d]_{\ell^n} & \G_m\ar[r]\ar[d]^{\id} & 1\\
 1\ar[r] & \m_{\ell^{n'}}\ar[r] & \G_m\ar[r]^{\ell^{n'}} & \G_m\ar[r] & 1.}$$
The induced map
$$\ell^n:\Het^2(Z,\m_{\ell^{n+n'}})\to\Het^2(Z,\m_{\ell^{n'}})$$
is identified with the reduction map $$\Het^2(Z,\Z_\ell(1)\tensor\Z/\ell^{n+n'}\Z)\to\Het^2(Z,\Z_\ell(1)\tensor\Z/\ell^{n'}\Z),$$
so it sends $\alpha_{n+n'}$ to $\alpha_{n'}$.  On the other hand, $\ell^n$ acts by multiplication by $\ell^n$ on $\Het^2(Z,\G_m)$.  The result follows from the resulting diagram of cohomology groups.
\end{proof}

\begin{lem}
 The Brauer classes associated to $\ell$-adic $B$-fields on $X$ form a subgroup $$\Br_\ell^B(Z)\subset\Br(Z)(\ell)$$ of the $\ell$-primary part of the Brauer group of $Z$.
\end{lem}
\begin{proof}
 Let $\alpha/\ell^n$ and $\beta/\ell^m$ be $B$-fields with Brauer classes $[\alpha_n]$ and $[\beta_m]$.  By Lemma \ref{L:reduction multiplication} we have that $[\alpha_n]=\ell^m[\alpha_{n+m}]$ and $[\beta_m]=\ell^n[\beta_{n+m}]$.  Let $\gamma=\ell^m\alpha+\ell^n\beta$.  We have that
 $$[\gamma_{n+m}]=\ell^m[\alpha_{n+m}]+\ell^n[\beta_{n+m}]=[\alpha_n]+[\beta_m],$$
 as desired.
\end{proof}

 Given a smooth projective geometrically connected algebraic surface $X$ over $k$, the intersection pairing defines a map
 $$\Het^2(X,\Z_\ell(1))\times\Het^2(X,\Z_\ell(1))\to\Het^4(X_{\ol{k}},\Z_\ell(2))=\Z_\ell.$$
 There is a cycle class map
 $$\Pic(X)\tensor\Z_\ell\to\Het^2(X,\Z_\ell(1)).$$
 Write $P(X,\Z_\ell)$ for its image, which is a $\Z_\ell$-sublattice.

\begin{defn}
 The \emph{$\ell$-adic transcendental lattice of $X$\/} is $$T(X,\Z_\ell):=P(X,\Z_\ell)^{\perp}\subset\Het^2(X,\Z_\ell(1))$$
\end{defn}

\begin{lem}\label{L:computation of Br}
The map
\begin{displaymath}
\beta \colon T(X,\Z_\ell)\tensor\Q_\ell/\Z_\ell \to \Br^B_\ell(X), \qquad \alpha \otimes (1/\ell^n) \mapsto [\alpha_n]
\end{displaymath}
is surjective with finite kernel.  It is an isomorphism if $\ell \nmid \disc(\Pic(X))$.
\end{lem}

\begin{proof}
 The map $\beta$ extends to a map \[
	 \beta:\Het^2(X,\Z_\ell(1))\tensor\Q_\ell\to\Br_\ell(X)
	 \] by sending $\alpha/\ell^n$ to the element $[\alpha_n]$ as above. One easily checks that this map is well-defined, and it is surjective by the definition of $\Br_\ell(X)$. Moreover, by construction \[P(X,\Z_\ell)\tensor\Q_\ell+\Het^2(X,\Z_\ell(1))\] lies in the kernel of $\beta$. Since \[
		 \Het^2(X,\Z_\ell(1))\tensor\Q_\ell=P(X,\Z_\ell)\tensor\Q_\ell\oplus T(X,\Z_\ell)\tensor\Q_\ell,
		 \] we see that the map defined in the statement of the Lemma is surjective. It remains to prove the assertions about its kernel.
 
Suppose $\alpha/\ell^n$ maps to $0$ in $\Br(X)$.  We have that $[\alpha_n]=0$, so that $$\alpha_n\in\Pic(X)/\ell^n\Pic(X)\subset\Het^2(X,\m_{\ell^n}).$$
Taking the cohomology of the exact sequence
$$0\to\Z_\ell(1)\to\Z_\ell(1)\to\m_{\ell^n}\to 0,$$
we see that $$\alpha\in P(X, \Z_{\ell}) + \ell^n\Het^2(X,\Z_\ell(1)).$$
Since $H^2(X, \Z_{\ell}(1))$ is unimodular, $P(X, \Z_{\ell}) \oplus T(X, \Z_{\ell})$ contains $\ell^{\nu} H^2(X, \Z_{\ell}(1))$, where $\nu$ is the $\ell$-adic valuation of $\disc(\Pic(X))$.  It follows that
\begin{displaymath}
(P(X, \Z_{\ell})+\ell^n H^2_{\et}(X, \Z_{\ell}(1))) \cap T(X, \Z_{\ell}) \subset
\ell^{n-\nu} T(X, \Z_{\ell}).
\end{displaymath}
We thus see that $\alpha \in \ell^{n-\nu} T(X, \Z_{\ell})$, and so $\ell^{\nu}$ kills $\alpha/\ell^n$ as an element of $T(X, \Z_{\ell}) \otimes \Q_{\ell}/\Z_{\ell}$.  It follows that the kernel of $\beta$ is contained in $T(X, \Z_{\ell}) \otimes \frac{1}{\ell^{\nu}} \Z_{\ell}/\Z_{\ell}$, which is always finite and is zero when $\nu=0$.
\end{proof}

\begin{prop}\label{P:B-fields explain it}
 If $X$ is a smooth projective surface over the finite field $k$ of characteristic $p$ then the following are equivalent.
\begin{enumerate}
 \item $\Br(X)$ is infinite
 \item $\Br^B_\ell(X)\neq 0$ for all $\ell$
 \item $\Br^B_\ell(X)\neq 0$ for some $\ell$
 \item $T(X,\Z_\ell)\neq 0$
\end{enumerate}
\end{prop}
\begin{proof}
By  \cite{LLR}, $\Br(X)$ is infinite if and only if $\Br(X)(\ell)$ is infinite for one $\ell$ if and only if $\Br(X)(\ell)$ is infinite for all $\ell$.  To prove the Proposition it suffices to prove that $\Br^B_\ell(X)\neq 0$ if $\Br(X)(\ell)$ is infinite.  So suppose $\alpha_n\in\Br(X)(\ell^n)$ is a sequence of classes.  Choose lifts
$$\widetilde\alpha_n\in\Het^2(X,\m_{\ell^n})$$ for each $n$.  Let $\beta_0=0\in\Het^2(X,\m_{\ell^0})$.  Assume we have constructed $\beta_1,\ldots,\beta_m$, $\beta_i\in\Het^2(X,\m_{\ell^i})$ such that $\ell\beta_{i+1}=\beta_i$ and $\ell^{j}\widetilde\alpha_{m+j}=\beta_m$ for all $j>0$.  The group $\Het^2(X,\m_{\ell^{m+1}})$ is finite, so there is one element $\beta_{m+1}$ that is a multiple of infinitely many $\widetilde\alpha_n$ such that $\ell\beta_{m+1}=\beta_m$.  Replacing the sequence of $\widetilde\alpha_n$ with the subsequence mapping to $\beta_{m+1}$, we see that we can proceed by induction, yielding an element
$$\beta\in\Het^2(X,\Z_\ell(1))$$
giving infinitely many distinct elements of $\Br^B_\ell(X)$.
\end{proof}


\subsection{Twisted $\ell$-adic Mukai lattices}

Fix a K3 surface $X$ over the field $k$ and an element $\alpha\in T(X,\Z_\ell)$.  Fix a $B$-field $\alpha/r$ with $r=\ell^n$ for some $n$.  Note that because $X$ is simply connected it follows from the Leray spectral sequence that $$\H^4(X,\Z_\ell(2))=\H^4(X_{\ol{k}},\Z_\ell(2))=\Z_\ell.$$

\begin{defn}
 The $\ell$-adic Mukai lattice of $X$ is the free $\Z_\ell$-module
 $$\Het(X,\Z_\ell):=\Het^0(X,\Z_\ell)\oplus\Het^2(X,\Z_\ell(1))\oplus\Het^4(X,\Z_\ell(2))$$
 with the intersection pairing
 $$(a,b,c)\cdot(a',b',c')=bb'-ac'-a'c\in\Het^4(X,\Z_\ell(2))=\Z_\ell.$$
\end{defn}
The algebraic part of the cohomology gives a sublattice
$$\Chow(X,\Z_\ell)=\Z_\ell\oplus P(X,\Z_\ell)\oplus\Z_\ell.$$  It is easy to see that $$\Chow(X,\Z_\ell)^\perp=T(X,\Z_\ell),$$ as sublattices of $\Het(X,\Z_\ell)$.  We will write $\Het(X,\Q_\ell)$ (resp.\ $\Chow(X,\Q_\ell)$, resp.\ $T(X,\Q_\ell)$) for $\Het(X,\Z_\ell)\tensor_{\Z_\ell}\Q_\ell$ (resp.\ $\Chow(X,\Z_\ell)\tensor\Q_\ell)$, resp.\ $T(X,\Z_\ell)\tensor\Q_\ell$).  The lattice $\Chow(X,\Z_\ell)$ has an integral structure: $$\Chow(X,\Z)=\Z\oplus\Pic(X)\oplus\Z$$ such that $\Chow(X,\Z)\tensor\Z_\ell=\Chow(X,\Z_\ell)$. We will also write $\Chow(X,\Q)$ for $\Chow(X,\Z)\tensor\Q$.

Following \cite{Y}, we consider the map
$$T_{-\alpha/r}:\Het(X,\Z_\ell)\to\Het(X,\Q_\ell)$$
that sends $x$ to the cup product $x\cup e^{-\alpha/r}$.

\begin{defn}
 The \emph{$\alpha/r$-twisted Chow lattice of $X$\/} is $$\Chow^{\alpha/r}(X,\Z_\ell):=(T_{-\alpha/r})^{-1}(\Chow(X,\Q_\ell)).$$
\end{defn}

There is also an integral structure on the twisted Chow lattice.

\begin{lem}
	\label{L:integral twisted Chow}
	Suppose $\alpha$ is primitive. Let \[
		\Chow^{\alpha/r}(X,\Z)=\left\{\left(ar,D+a\alpha,c\right)\ |\ a,c\in\Z, D\in \Pic(X)\right\}\subset T_{-\alpha/r}^{-1}(\Chow(X,\Q))\subset\Chow^{\alpha/r}(X,\Z_\ell).\] The natural map 
\[
	\Chow^{\alpha/r}(X,\Z)\tensor\Z_\ell\to\Chow^{\alpha/r}(X,\Z_\ell)
	\] is an isomorphism.
\end{lem}
\begin{proof}
	A basis for $\Chow^{\alpha/r}(X,\Z)$ is given by $(r,\alpha,0)$, the vectors $(0,b,0)$ where $b$ ranges over a basis for $\Pic(X)$, and $(0,0,1)$. Since $\Pic(X)$ is a lattice, we see that these vectors all remain linearly independent over $\Z_\ell$ (inside $\Het(X,\Z_\ell)$), so that the displayed scalar extension map is injective. To see surjectivity, suppose $(x,y,z)\in\Het(Z,\Z_\ell)$ is an element such that $e^{-\alpha/r}\cup(x,y,z)\in\Chow(X,\Q_\ell)$. Computing the components of the cup product shows that $(x,y,z)$ must satisfy the conditions: $x,y\in\Z_\ell$ (via the natural identifications $\H^0(X,\Z_\ell)=\Z_\ell$ and $\H^4(X,\Z_\ell(2))=\Z_\ell)$ and $-\alpha x/r+y\in\Pic(X)\tensor\Q_\ell$. Since $\alpha$ is primitive and orthogonal to $\Pic(X)$ and $y\in\H^2(X,\Z_\ell)$, we see that we must have $x=ar$ for some $a$ and then $-a\alpha+y\in\Pic(X)\tensor\Z_\ell$, so $y=a\alpha+D$ for some $D\in\Pic(X,\Z_\ell)$. This shows that the displayed map is surjective.
\end{proof}

%

The (integral) twisted Chow lattice is a natural recipient of Chern classes for twisted sheaves.  Let $\pi:\ms X\to X$ be a $\m_r$-gerbe representing the class $-[\alpha_n]$ associated to the $B$-field $-\alpha/r$, $r=\ell^n$.  Suppose $\ms P$ is a perfect complex of $\ms X$-twisted sheaves of positive rank.

\begin{defn}
 The \emph{Chern character of $\ms P$\/} is the unique element
 $$\chern_{\ms X}(\ms P)\in\Chow(X,\Q)$$
 such that $$\rk\chern_{\ms X}(\ms P)=\rk\ms P$$ and $$\chern_{\ms X}(\ms P)^r=\chern\left(\R\pi_\ast(\ms P^{\ltensor r})\right)\in\Chow(X,\Q).$$
\end{defn}

\begin{remark}
A less \emph{ad hoc\/} approach is to define rational Chern classes using a splitting principle,
etc.  This can be done and yields the same result (and works for complexes of rank $0$).
In particular, the Chern character defined here satisfies the Riemann-Roch theorem in the following
sense: given two perfect complexes $\ms P$ and $\ms Q$ of $\ms X$-twisted sheaves, we can compute
$\chi(\shom(\ms P,\ms Q))$ as $\deg(\chern(\ms P^{\vee})\cdot\chern(\ms Q)\cdot\Todd_X)$.
\end{remark}

\begin{defn}\label{D:Mukai vector}
 Given a perfect complex $\ms P$ of $\ms X$-twisted sheaves as above, we define the \emph{twisted Mukai vector of $\ms P$\/} to be
 $$v^{\alpha/r}(\ms P):=e^{\alpha/r}\chern_{\ms X}(\ms P)\sqrt{\Todd_X}.$$
\end{defn}

\begin{lem}\label{L:chern char props}
 Given perfect complexes $\ms P$ and $\ms Q$ of $\ms X$-twisted sheaves as above, the following hold.
\begin{enumerate}
 \item the element $v^{\alpha/r}(\ms P)$ lies in the integral subring $\Chow^{\alpha/r}(X,\Z)$
 \item we have that
 $$\chi(\ms P,\ms Q)=-v^{\alpha/r}(\ms P)\cdot v^{\alpha/r}(\ms Q).$$
\end{enumerate}
\end{lem}
\begin{proof}
 To see the integrality, we may assume that the base field is algebraically closed. Since the calculation of $v^{\alpha/r}$ factors through numerical $K$-theory, it is enough to show the analogous result for a class in the integral (numerical) $K$-theory of twisted sheaves. In particular, since $X$ satisfies the resolution property \cite[Corollary 2.2.7.21]{L1}, we may assume that $\ms P$ is a locally free $\ms X$-twisted sheaf.

Given a sheaf $\ms F$ on $\ms X$, write $\Def_{X, \ms{F}}$ for the deformation functor parametrizing pairs of a deformation of $X$ and a deformation of $\ms{F}$ over the induced deformation of $\ms{X}$. By \cite[Lemma~5.2.8]{L0}, there is a finite-colength subsheaf $\ms P'\subset\ms P$ such that the trace map
$$\Ext^2(\ms{P}', \ms{P}') \to \H^2(X, \ms{O})$$
is an isomorphism. Applying \cite[Proposition~2.2.4.9]{L1}, this implies that the morphism of formal deformation functors
$$\Def_{X,\ms{P}'} \to \Def_{X,\det,\ms{P}'}$$
is formally smooth. In particular, $\ms{P}'$ is unobstructed on any infinitesimal deformation over which its determinant deforms. Since $\ms{X}$ is smooth, the morphism of deformation functors
$$\Def_{\ms{P}/\ms{P}'} \to \Def_X$$
is also formally smooth.

Replacing $\ms P$ by $\ms P'\oplus\ms P/\ms P'$, we may thus assume that the functor
$$\Def_{X,\ms{P}} \to \Def_{X,\det,\ms{P}}$$
is formally smooth. Note that $\det\ms P$ is naturally the pullback of an invertible sheaf on $X$, since $\ms P$ has rank divisible by $\ell^n$ ($\alpha$ being primitive by assumption). 
We can thus work with the universal deformation of $X$ over which a chosen ample divisor $H$ and $\det\ms P$ remain algebraic, so that $\ms P$ also deforms. By Grothendieck Existence for twisted sheaves \cite[Lemma~2.3.1.3]{L1}, the formal deformation of $\ms{P}$ algebraizes, so we see that we may prove the result under the assumption that the base field has characteristic 0. Since the integrality is invariant under algebraically closed base field extension, we may assume that the base field is $\C$. Lemma 3.1, Lemma 3.3, and Remark 3.2 of \cite{Y} then imply that $v^{\alpha/r}(\ms P)$ has precisely the form described in Lemma \ref{L:integral twisted Chow}.

To see the second statement, note that
\begin{align*}
-v^{\alpha/r}(\ms P)\cdot v^{\alpha/r}(\ms Q)&=
\deg\left(\chern_{\ms X}(\ms P^\vee\ltensor\ms Q)\Todd_X\right)\\
&=\deg\left(\chern(\R\pi_\ast(\rshom(\ms P,\ms Q))\Todd_X\right)\\
&=\chi(\R\pi_\ast\rshom(\ms P,\ms Q))\\
&=\chi(\ms P,\ms Q)
\end{align*}

\end{proof}

\subsection{Moduli spaces of twisted sheaves}
We recall the essential details of the moduli theory of twisted sheaves on K3 surfaces that we will need, building on the foundational work in \cite{M}.  The reader can consult \cite{Y} for a development of a more general theory using a stack-free formulation of the notion of twisted sheaf. In particular, we only consider the case in which the rank equals the order of the Brauer class in order to simplify the exposition; \cite{Y} contains the general theory (over $\C$). 

Fix a K3 surface $X$ over $k$, $\alpha/\ell^n$ a primitive $\ell$-adic $B$-field, and $v=(r,D,s)\in\Chow^{\alpha/\ell^n}(X,\Z)$ a Mukai vector.

\begin{defn}
 The \emph{stack of simple $\ms X$-twisted sheaves with Mukai vector $v$} is the stack $\ms M_{\ms X}(v)$ whose objects over a $k$-scheme $T$ are pairs $(\ms F,\phi)$, where $\ms F$ is a $T$-flat quasi-coherent $\ms X_T$-twisted sheaf of finite presentation and $\phi:\det\ms F\simto\ms O(D)$ is an isomorphism of invertible sheaves on $\ms X$, such that for every geometric point $t\to T$, the fiber sheaf $\ms F_t$ has Mukai vector $v$ and endomorphism ring $k(t)$.
\end{defn} 

Using Artin's Representability Theorem, one can show that $\ms M_{\ms X}(v)$ is an Artin stack locally of finite type over $k$. This is proven in Section 2.3.1 of \cite{L1}. 

\begin{prop}\label{P:moduli summary}
 Given an $\ell$-adic $B$-field $\alpha/\ell^n$ and a primitive Mukai vector $$v\in\Chow^{\alpha/\ell^n}(X,\Z)$$ such that $\rk v=\ell^n$ and $v^2=0$, we have that
\begin{enumerate}
 \item the stack
 $\ms M_{\ms X}(v)$
of simple $\ms X$-twisted sheaves with Mukai vector $v$ is a $\m_{\ell^n}$-gerbe over a K3 surface $M_{\ms X}(v)$;
\item the universal sheaf $\ms E$ on $\ms X\times\ms M_{\ms X}(v)$ defines an equivalence of derived categories
$$\D^{\tw}(\ms X)\simto\D^{-\tw}(\ms M_{\ms X}(v));$$
\item if there is a vector $u\in\Chow^{\alpha/\ell^n}(X,\Z)$ such that $(u\cdot v,\ell)=1$ then the Brauer class of the gerbe
$$\ms M_{\ms X}(v)\to M_{\ms X}(v)$$
is trivial.  In particular, there is an equivalence of derived categories
$$\D^{-\tw}(\ms M_{\ms X}(v))\simto\D(M_{\ms X}(v)).$$
\end{enumerate}
\end{prop}
\begin{proof}
	It suffices to prove the first statement after extending scalars to $\widebar k$.
	First note that since the sheaves in question have rank $\ell^n$ and the order of the Brauer class is $\ell^n$, every object is automatically stable with respect to any polarization: any subsheaf has rank either $0$ or $\ell^n$, so stability conditions become vacuous. (See, e.g., Sections 2.2.5 and 3.2.4 of \cite{L1} for further details.)

	Let $L$ be the second component of $v$. Tensoring the sheaves parametrized by $\ms M_{\ms X}(v)$ by an invertible sheaf $M$ results in an isomorphic moduli problem parametrizing sheaves with determinant $L\tensor M^{\tensor\ell^n}$. Thus, taking $M\in\Pic(X)\setminus p\Pic(X)$ to be sufficiently ample we see that we can assume that $L$ is ample and not contained in $p\Pic(X)$. 

	By Deligne's theorem \cite{D2} and the smooth and proper base change theorems in \'etale cohomology, there is a pointed connected smooth $W(k)$-scheme $(M,m)$, a projective relative K3 surface $\mc X\to M$, a class $\widetilde v=(\ell^n,\ms L,c)\in\Chow(\mc X)$, a class $\widetilde\alpha\in\H^2(\mc X,\Z_\ell(1))$, and an isomorphism between the fiber of the triple $(\mc X,\widetilde v,\widetilde\alpha)$ over $m$ and the triple $(X,v,\alpha)$. Moreover, the geometric generic fiber of $\mc X\to M$ has Picard group generated	by the restriction of $\ms L$. We choose a Henselian discrete valuation ring $R$ with algebraically closed residue field dominating $M$ and centered at $m$ and pull the family back to $R$. This results in a triple $(\mc X_R,\widetilde v_R,\widetilde\alpha_R)$ with closed fiber $(X,v,\alpha)$ (up to algebraically closed base field extension) and geometric generic fiber having Picard number $1$ and characteristic $0$, on which $v$ remains primitive. Write $H$ for the relative polarization from $\ms L$ and $v$ for $\widetilde v_R$. 
	
	Consider the relative moduli space $\ms M(v)\to\spec R$ parametrizing $H$-stable $\ms X$-twisted sheaves with Mukai vector $v$. Since any simple twisted sheaf with unobstructed determinant is unobstructed, we know that $\ms M(v)$ is a $\m_{\ell^n}$-gerbe over a smooth algebraic space $\mc M\to\spec R$. By Theorem 3.16 of \cite{Y}, the geometric generic fiber of $\mc M$ over $R$ is a K3 surface. By Lemma 2.3.3.2 of \cite{L1}, the algebraic space $\mc M$ is proper over $R$. It follows that $\mc M\to R$ is a relative K3 surface, as claimed. In particular, the special fiber is a K3 surface.

	By the foundational results on Fourier-Mukai transforms (see, for example, \cite{C}), the second statement is equivalent to the adjunction maps \begin{equation}\label{E:e}
		\ms O_{\Delta_{\ms X}}\to\R(\pr_{13})_\ast\left(\L\pr_{12}^\ast\ms E\ltensor\L\pr_{23}^\ast\ms E^{\vee}\right)
	\end{equation} and \begin{equation}\label{E:e2}
			\R(\pr_{13})_\ast\left(\L\pr_{12}^\ast\ms E\ltensor\L\pr_{23}^\ast\ms E^{\vee}\right)\to \ms O_{\Delta_{\ms X}}
		\end{equation} being quasi-isomorphisms. (See Proposition 3.3 of \cite{LO} for a proof without any assumptions on the base field.) It thus suffices to prove the result after base change to $\widebar k$. In this case, the proof proceeds precisely as in section 5.2 (proof of Theorem 1.2) of \cite{C}: one verifies the classical criterion of Bridgeland for the Fourier-Mukai functor to give an equivalence of derived categories.

 We now return to the situation in which our base field is the original finite field $k$. To prove the last statement, it suffices to prove the following lemma.
\begin{lem}\label{L:kill class}
 Given a Mukai vector $$u\in\Chow^{\alpha/\ell^n}(X,\Z),$$ there is a perfect complex $\ms P$ of $\ms X$-twisted sheaves such that $$v^{\alpha/\ell^n}(\ms P)=u.$$
\end{lem}
Let us accept the lemma for a moment and see why this implies the result.  Write $\ms E$ for the universal sheaf on $$\ms X\times\ms M_{\ms X}(v),$$ and let $p$ and $q$ denote the first and second projections of that product, respectively. 
The sheaf $\ms E$ is simultaneously $\ms X$- and $\ms M_{\ms X}(v)$-twisted.  Given a complex $\ms P$ as in Lemma \ref{L:kill class}, consider the perfect complex of $\ms M_{\ms X}(v)$-twisted sheaves
$$\ms Q:=\R q_\ast(\L p^\ast\ms P^\vee\ltensor\ms E).$$
The rank of this complex over a geometric point $m$ of $\ms M$ is calculated by
$$\chi(\ms P,\ms E_m)=-v^{\alpha/\ell^n}(\ms P)\cdot v=-u\cdot v,$$
which is relatively prime to $\ell$.  By standard results, we have that the Brauer class of $\ms M_{\ms X}(v)$ satisfies $$[\ms M_{\ms X}(v)]\in\Br(M_{\ms X}(v))[u\cdot v].$$
On the other hand, $\ms M_{\ms X}(v)\to M_{\ms X}(v)$ is a $\m_{\ell^n}$-gerbe, which implies that
$$[\ms M_{\ms X}(v)]\in\Br(M_{\ms X}(v))[\ell^n].$$
Combining the two statements yields the result.

It remains to prove Lemma \ref{L:kill class}.
\begin{proof}[Proof of Lemma \ref{L:kill class}]
We know that
$$u=(ra,D+a\alpha,c),$$ so we seek a perfect complex of rank $ra$, determinant $D$, and appropriate second Chern class.  By Theorem 4.3.1.1 of \cite{L2}, there is a locally free $\ms X$-twisted sheaf $V$ of rank $r$.  Moreover, since any curve over a finite field has trivial Brauer group the restriction $\ms X\times_X C$ has trivial Brauer class, hence supports twisted sheaves $L_C$ of rank $1$ by Lemma 3.1.1.8 of \cite{L2}.  Since $\det L_C\cong\ms O(C)$, we see that we can get any determinant by adding (in the derived category) a sum of shifts of invertible sheaves supported on curves.  Finally, Wedderburn's theorem yields invertible twisted sheaves supported at any closed point $p$ of $X$, and they have twisted Mukai vector $(0,0,d)$, where $d$ is the degree of $p$. By the Lang-Weil estimates \cite[Corollary 3]{LW} $X$ has a $0$-cycle of degree $1$, so there is a sum of (shifted) twisted skyscraper sheaves giving the desired second Chern class.  Since any bounded complex on $\ms X$ is perfect ($\ms X$ being regular), the lemma is proven.
\end{proof}

This completes the proof of Proposition \ref{P:moduli summary}.
\end{proof}

\begin{remark}
	We note that Theorem 4.3.1.1 of \cite{L2} is far from trivial; it relies on de Jong's period-index theorem and a careful analysis of the asymptotic geometry of moduli spaces of twisted sheaves on surfaces (in particular, their asymptotic irreducibility). In this sense, the last part of Proposition \ref{P:moduli summary}, giving a ``numerical'' criterion linking a twisted derived category to an untwisted one, is the deepest. As we will see below, this equivalence is the glue that holds this direction of the argument together.
\end{remark}

\subsection{Twisted partners of K3 surfaces over a finite field}

Fix a K3 surface $X$ over the finite field $k$ of characteristic $p\geq 5$.  Let $k'/k$ be the quadratic extension
of $k$.

\begin{lem}\label{L:disc}
Let $V$ be a non-degenerate quadratic space over a field $F$ and let $\phi$ be a semi-simple element of the orthogonal
group of $V$, all of whose eigenvalues belong to $F$.  Then
\begin{displaymath}
\disc(V^{\phi=1} \oplus V^{\phi=-1})=(-1)^{n/2} \disc(V),
\end{displaymath}
where $n=\dim(V)-\dim(V^{\phi=1} \oplus V^{\phi=-1})$.
\end{lem}

\begin{proof}
Write $V_{\lambda}$ for the $\lambda$ eigenspace of $\phi$.  Suppose $\lambda \ne \pm 1$.  Since $\phi$
preserves the form, elements of $V_{\lambda}$ pair to zero with elements of $V_{\mu}$ unless $\mu=\lambda^{-1}$.
Note that $\lambda^{-1} \ne \lambda$.  The discriminant of $V_{\lambda} \oplus V_{\lambda^{-1}}$ is $(-1)^{d/2}$, where
$d=\dim(V_{\lambda})$, as is easily seen by choosing a basis for $V_{\lambda}$, taking the dual basis for
$V_{\lambda^{-1}}$ and computing the matrix of pairings.  Now, list the eigenvalues of $\phi$, other than $\pm 1$,
as $\lambda_1, \lambda_1^{-1},
\ldots, \lambda_k, \lambda_k^{-1}$.  Then $V$ decomposes as a direct sum of $V_1$, $V_{-1}$ and the $V_{\lambda_i}
\oplus V_{\lambda_i^{-1}}$.  We thus find that $\disc(V)=\disc(V_1 \oplus V_{-1}) (-1)^{n/2}$, as was to be shown.
\end{proof}

\begin{lem}\label{L:non-zero square}
Suppose that $\Br(X_{k'})$ is infinite, and let $d$ be a rational number.  Then there are infinitely many
primes $\ell$ for which $T(X_{k'}, \Z_{\ell})$ contains an element $\gamma$ with $\gamma^2=d$.
\end{lem}

\begin{proof}
Let $\phi'=\phi^2$ be the Frobenius over $k'$.  If $\ell$ does not divide the discriminant of $\Pic(X_{k'})$ then the discriminant of $P(X_{k'}, \Z_{\ell})$ is an $\ell$-adic unit, and we have an orthogonal
decomposition
\begin{equation}
\label{eq5}
H^2_{\et}(X_{\ol{k}}, \Z_{\ell}(1))^{\phi'=1}=P(X_{k'}, \Z_{\ell}) \oplus T(X_{k'}, \Z_{\ell}).
\end{equation}
By Proposition~\ref{step1a}, the discriminant of the lattice $H^2_{\et}(X_{\ol{k}}, \Z_{\ell}(1))^{\phi'=1}$ is an
$\ell$-adic unit for $\ell \gg 0$.  We thus find that the discriminant of $T(X_{k'}, \Z_{\ell})$ is an $\ell$-adic unit
for $\ell \gg 0$.  Therefore, if $T(X_{k'}, \Z_{\ell})$ has rank at least two then any element of $\Z_{\ell}$ is of
the form $\gamma^2$ for some $\gamma \in T(X_{k'}, \Z_{\ell})$.  (This follows from \cite[\S IV.1.7, Prop.~4]{Serre} and
Hensel's lemma.)

We must now handle the case where the transcendental lattice has rank one.  Let $E$ be the number field generated by
the eigenvalues of $\phi$ and the square roots of $-1$, $d$ and the discriminant of $\Pic(X_{k'})$.  Let $\ell \gg 0$
be a large prime which splits completely in $E$.  Recall from \cite{D} that the action of $\phi$ on
$H^2_{\et}(X_{\ol{k}}, \Q_{\ell}(1))$ is semi-simple.  We have
\begin{displaymath}
H^2_{\et}(X_{\ol{k}}, \Q_{\ell}(1))^{\phi'=1}=
H^2_{\et}(X_{\ol{k}}, \Q_{\ell}(1))^{\phi=1} \oplus H^2_{\et}(X_{\ol{k}}, \Q_{\ell}(1))^{\phi=-1}
\end{displaymath}
Lemma~\ref{L:disc} thus shows that the discriminant of $H^2_{\et}(X_{\ol{k}}, \Q_{\ell}(1))^{\phi'=1}$ is $\pm 1$, and
therefore a square in $\Q_{\ell}$; the same is true when we use $\Z_{\ell}$ coefficients.
Since the discriminant of $\Pic(X_{k'})$ is also a square in $\Q_{\ell}$, the equation~\eqref{eq5} shows that the
discriminant of $T(X_{k'}, \Z_{\ell})$ is a square in $\Q_{\ell}$.  As in the previous paragraph, we know this
discriminant is also a unit.  Since the transcendental lattice has rank 1, we thus find that it has an element
$\alpha$ such that $\alpha^2=1$.  Since $d$ is a square in $\Q_{\ell}$, we can find a multiple $\gamma$ of $\alpha$
with $\gamma^2=d$.
\end{proof}

Now assume that $X$ is a K3 surface over the finite field $k$ such that $\Br(X)$ is infinite.  Fix a primitive
ample divisor $D\in\Pic(X)$ and choose an odd prime $\ell$ relatively prime to $\disc\Pic(X)$ and $D^2$ and an
element $\gamma \in
T(X_{k'}, \Z_{\ell})$ such that $\gamma^2=-D^2$.  (This is possible by the above lemma.)

\begin{defn}
 Given a relative K3 surface $\mc Y\to S$, an invertible sheaf $\ms L\in\Pic(\mc Y)$, and positive integers $r$ and $s$,
let $$\Sh_{\mc Y/S}(w)\to S$$ be the stack of simple locally free sheaves with Mukai vector $w=(r,\ms L,s)$ on each fiber.
\end{defn}

\begin{remark}
  A family in this moduli problem over $T$ consists of a locally free sheaf $V$ of constant rank on $\mc Y_T$ and an isomorphism $$\det V\simto\ms L_{\mc Y_T}$$ such that for each geometric point $t\to T$, the restriction $V_t$ has Mukai vector $w$.
\end{remark}

\begin{lem}\label{L:relative moduli}
 If $w$ is primitive and $w^2=0$ then $\Sh_{\mc Y/S}(w)$ is a $\m_r$-gerbe over a smooth algebraic space $\mSh_{\mc Y/S}(w)$ of relative dimension $2$ over $S$ with non-empty geometric fibers.
\end{lem}
\begin{proof}
	Non-emptiness of the geometric fibers is proven in Proposition 3.15 of \cite{Y}.  The smoothness follows from the fact that that obstruction theory of a locally free sheaf $V$ with determinant $\ms L$ on a fiber $\mc Y_s$ is given by the kernel of the trace map $$\Ext^2(V,V)\to\Het^2(\mc Y_s,\ms O),$$
 which vanishes when $V$ is simple.  The relative dimension is
 $$\dim\Ext^1(V,V)=w^2+2=2,$$
 as desired.
\end{proof}
When $S=\spec L$ and the base field is understood, we will write simply $\Sh_{\mc Y}(w)$ and $\mSh_{\mc Y}(w)$.

\begin{prop}\label{P:partner blob}
 In the above situation, there is an infinite sequence of pairs $(\gamma_n,M_n)$ such that
\begin{enumerate}
 \item $\gamma_n\in\Het^2(X_{k'}, \m_{\ell^n})$ with $[\gamma_n]$ of exact order $\ell^n$ in $\Br(X_{\widebar k})$;
 \item $M_n$ is a K3 surface over $k'$ such that $\rho(M_n)=\rho(X_{k'})$ and $$\rk T(M_n,\Z_\ell)=\rk T(X_{k'},\Z_\ell);$$
 \item given a gerbe $\ms X_n\to X_{k'}$ representing $\gamma_n$, we have that $M_n$ is a fine moduli space of stable
$\ms X_n$-twisted sheaves with determinant $D$;
 \item for each $n$, there is a (classical) Mukai vector $w_n\in\Chow(M_n)$ such that
\begin{enumerate}
 \item a tautological sheaf $\ms E$ on $\ms X_n\times M_n$ induces an open immersion $$\ms X_n\inj \Sh_{M_n}(w_n);$$
 \item given a complete dvr $R$ with residue field $k''$ containing $k'$ and a relative K3 surface $\mc M$ over $R$ with
$\mc M\tensor_R k''=M_n\tensor_{k'} k''$ such that $\Pic(\mc M)=\Pic(M_n),$ there is a relative K3 surface $\mc X\to\spec R$ such that
$\mc X\tensor_R k''=X$, a $\m_{\ell^n}$-gerbe $\mf X\to\mc X$ such that $\mf X\tensor_R K=\ms X_n$, and a locally free
$\mf X\times \mc M$-twisted sheaf $\mf E$ such that $\mf E\tensor_R k''=\ms E$.  For any inclusion $R\inj\C$, the base
changed family $(\mf X_{\C},\mc M_{\C}, \mf E_{\C})$ yields a twisted Fourier-Mukai partnership
\begin{displaymath}
(\mc X_{\C},[\mf X_{\C}])\sim \mc M_{\C},
\end{displaymath}
where $[\mf X_{\C}]\in\Br(\mc X_{\C})$ is a Brauer class of exact order $\ell^n$.
\end{enumerate}
\end{enumerate}
\end{prop}

\begin{proof}
Consider the Mukai vector
\begin{displaymath}
v_n:=(\ell^n,\gamma+D,0)\in\Chow^{\gamma/\ell^n}(X_{k'},\Z).
\end{displaymath}
We have that
$$v_n^2=\gamma^2+D^2=2q^2+D^2=0$$
and
$$v_n\cdot(\ell^n,\gamma,0)=\gamma^2=2q^2=-D^2\not\equiv 0\mod\ell.$$
It is easy to see that $\gamma$ is primitive.  Consider the diagram
$$\xymatrix{0\ar[r] & \Pic(X_{k'})\tensor\Z_\ell\ar[r]\ar[d] & \H^2(X_{k'},\Z_\ell(1))\ar[r]\ar[d] & T_\ell\Br(X_{k'})\ar[r]\ar[d] & 0\\
0\ar[r] & \Pic(X_{\widebar k})\tensor\Z_\ell\ar[r] & \H^2(X_{\widebar k},\Z_\ell(1))\ar[r] & T_\ell\Br(X_{\widebar k})\ar[r] & 0,
}$$
where $T_\ell$ denotes ``$\ell$-adic Tate module''.

We claim that the left square is Cartesian.  Indeed, it follows from the Hochschild-Serre spectral sequence that the middle vertical arrow identifies the top group with the Frobenius-invariants in the bottom group.  Similarly, since the Brauer obstruction to the existence of invertible sheaves on proper varieties over finite fields is trivial, the left vertical arrow is also the group invariants.  But the invariants in a submodule are just the intersection with the ambient invariants, showing that the square is Cartesian.  It follows that for each $n$, the image of $\gamma_n$ in $\Br(X_{\widebar k})$ has order exactly $\ell^n$, as claimed.

Let $\ms X_n$ be a gerbe representing $-\gamma_n$ (see Notation \ref{N:n} for the definition of $\gamma_n$ in terms of $\gamma$).  Applying Proposition \ref{P:moduli summary}, we have that $\ms M_{\ms X_n}(v_n)$ is a $\m_{\ell^n}$-gerbe over a K3 surface $M_n$ whose Brauer class is killed by $$v_n\cdot(\ell^n,\gamma,0)=-D^2,$$ hence is trivial.  This yields an equivalence
$$\D^{\tw}(\ms X_n)\simto\D(M_n).$$
Finally, the universal sheaf $\ms E$ on $\ms X_n\times\ms M_{\ms X_n}(v_n)$ induces an isometry
$$\Het(X_{k'},\Q_\ell)\simto\Het(M_n,\Q_\ell)$$
that restricts to an isometry
$$\Chow(X_{k'},\Q_\ell)\simto\Chow(M_n,\Q_\ell)$$
(see Section 4.1.1 of \cite{Y} for the proof, written using Yoshioka's notation),
establishing the first three statements of the Proposition.

Let $\ms E$ be a tautological sheaf on $\ms X_n\times M_n$.  The determinant of $\ms E$ is naturally identified with the pullback of an invertible sheaf $$L_{X_n}\boxtimes L_{M_n}\in\Pic(X_n\times M_n),$$ and each geometric fiber $\ms E_{x}$ has a second Chern class $c\in\Z$.  Letting $$w_n=(\ell^n,L_{M_n},c),$$ the sheaf $\ms E$ gives a morphism of $\m_{\ell^n}$-gerbes
$$\ms X_n\to \Sh_{M_n}(w_n).$$
Since $\ms E$ is a Fourier-Mukai kernel, this morphism is an \'etale monomorphism (see e.g.\ Section 4 of \cite{LO}), so it is an open immersion.

To prove the last part, we may replace $k$ with $k''$ and ignore the residue field extensions. Fix a lift $\mc M$ of $M_n$ over $R$ over which all of $\Pic(M_n)$ lifts (by assumption). Since $\Pic(M_n)$ lifts, so does $w_n$ and $\Sh_{\mc M/R}(w_n)$ is a $\m_{\ell^n}$-gerbe over a smooth algebraic space $\mSh_{\mc M/R}(w_n)$ over $R$.  Write $\ms V$ for the universal sheaf on $\Sh_{\mc M/R}(w_n)\times_R\mc M$.  We can write $$\det\ms V=\ms U\boxtimes L_{M_n}$$ with $$\ms U\in\Pic(\mSh_{\mc M/R}(w_n)),$$ and by assumption the pullback of $\ms U$ along the map
$$X\to\mSh_{M_n}(w_n)\inj\mSh_{\mc M/R}(w_n)$$
is isomorphic to $D$, hence is ample (by the assumption on $D$).  Since $\ms X_n$ is open in $\Sh_{M_n}(w_n)$, the induced open formal substack $$\ms Z\subset\widehat{\Sh}_{\mc M/R}(w_n)\to\spf R$$ is a $\m_{\ell^n}$-gerbe over a formal deformation $\mf Z$ of $X$ over $\spf R$.  The determinant of the universal formal sheaf gives a formal lift of $D$ over $\mf Z$, hence a polarization.  It follows that $\ms Z\to\mf Z$ is algebraizable, giving a $\m_{\ell^n}$-gerbe over a relative K3 surface $$\ms G\to\mc X\to\spec R$$ and an open immersion
$$\ms G\inj\Sh_{\mc M/R}(w_n)$$
extending $$\ms X_n\inj\Sh_{M_n}(w_n).$$  Restricting the universal sheaf gives the desired deformation $\mf E$ of $\ms E$.  Nakayama's lemma shows that $\mf E$ is the kernel of a relative Fourier-Mukai equivalence (i.e., the adjunction maps \eqref{E:e} and \eqref{E:e2} are quasi-isomorphisms bceause their derived restrictions to the fiber are quasi-isomorphisms).  Foundational details are contained in Section 3 of \cite{LO}, and a similar deformation argument is contained in Section 6 of \cite{LO}.

\end{proof}

\subsection{Proof of Main Theorem (2)}

Fix a finite field $k$ of characteristic $p$, and let $k'$ be the quadratic extension of $k$.  We now prove the
second part of the main theorem from the introduction.

\begin{thm}\label{T:finiteness implies tate}
Suppose $p \ge 5$.
If there are only finitely many K3 surfaces over $k'$ then for any K3 surface $X$ over $k$,
the Brauer group $X_{k'}$ is finite, i.e., the Tate conjecture holds for $X_{k'}$.
\end{thm}

\begin{proof}
Suppose that $X$ is a K3 surface over $k$ such that $\Br(X_{k'})$ is infinite.  We know by
Proposition~\ref{P:partner blob} that there is an infinite sequence of Brauer classes $\gamma_n\in\Br(X_{k'})$ such that
\begin{enumerate}
\item $\gamma_n$ has exact order $\ell^n$ over $\widebar {k}$, and
 \item the twisted K3 surface $(X_{k'},\gamma_n)$ is Fourier-Mukai equivalent to a K3 surface $M_n$ over $k'$.
\end{enumerate}
If there are only finitely many K3 surfaces over $k'$, there is a subsequence $n_i$ such that all of the $M_{n_i}$ are
the same surface, say $M$.  Since $M$ has infinite Brauer group (by Proposition \ref{P:partner blob}(2)), it is 
not Shioda-supersingular. Thus, $M$ has Picard number at most $4$ by Theorem 1.7 of \cite{A} (see \cite[\S 3]{LieblichMaulik}). Proposition 3.1 and Corollary 4.2 of \cite{LieblichMaulik} show that 
\begin{enumerate}
\item $M$ admits a deformation over $k'[\![t]\!]$ over which all of $\Pic(M)$ deforms, whose generic fiber has finite height;
\item any finite height K3 surface admits a lift to characteristic $0$ over which its entire Picard group deforms.
\end{enumerate}
Suppose $L\in\Pic(M)$ is ample of degree $2d$. Let $\ms D$ denote the finite-type DM stack parametrizing pairs $(X,\Lambda)$ with $\Lambda$ ample of degree $2d$, so that $(M,L)$ is a $k'$-point of $\ms D$. There is relative scheme $\ms I\to\ms D$ locally of finite type defined by sending a pair $(X,\Lambda)$ to $\operatorname{Isom}(N,\Pic(X))$, where $N=\Pic(M)$ and the isomorphisms are lattice isomorphisms.  

Let $\ms O$ be the strict local ring of the object $(M,\id)$ in $\ms I$; since $\ms I$ is locally of finite type over $W(k')$, the ring $\ms O$ is the Henselization of a ring essentially of finite type over $W(k')$ and is thus Noetherian. Moreover, the two statements above imply that $\ms O$ has a point of characteristic $0$ (that necessarily specializes to the unique closed point). It follows that there exists a lift $\mc M$ of $M$ over a complete dvr $R$ with residue field algebraic over 
$k'$ such that $\Pic(\mc M)=\Pic(M)$.  Since the residue field is algebraic over $k'$, there is an embedding $R\inj\C$. Applying Proposition \ref{P:partner blob}(4)(b) we end up with a
complex K3 surface $\mc M_{\C}$ with a sequence of twisted derived partners $(X_i,\eta_i)$ with $\eta_i\in\Br(X_i)$
of exact order $\ell^{n_i}$ for a strictly increasing sequence $n_i$.  In particular, $\mc M_{\C}$ has infinitely many
twisted partners.  But this is a contradiction by Corollary 4.6 of \cite{HS}.
\end{proof}

\hskip\baselineskip

\end{document}